\documentclass[11pt]{amsart}

\usepackage[a4paper,left=2.5cm,right=2.5cm,top=2.5cm,bottom=2.5cm]{geometry}
\usepackage[utf8]{inputenc}
\usepackage{amsfonts}
\usepackage{amssymb}
\usepackage{amsthm}
\usepackage{amsmath}
\usepackage{mathrsfs}
\usepackage{epsfig}
\usepackage{comment}
\usepackage{mathpazo}
\usepackage{tikz}
\usepackage{fp,ifthen}
\usepackage{esint}
\usepackage{nicefrac}
\usetikzlibrary{decorations.pathreplacing}
\usepackage{float}
\usepackage{dsfont}
\usepackage{enumerate}
\usepackage[colorlinks=true,linkcolor=blue,citecolor=blue,urlcolor=blue]{hyperref}
\usepackage[protrusion=true,expansion=true,stretch=30,shrink=30,step=1]{microtype}

\newtheorem{theorem}{Theorem}[section]
\newtheorem{lemma}[theorem]{Lemma}
\newtheorem{rmk}[theorem]{Remark}

\newcommand{\R}{\mathbb R}
\newcommand{\HH}{\mathcal H}

\newcommand{\A}{\mathcal A}
\newcommand{\PP}{\mathcal P}
\DeclareMathOperator{\tr}{tr}
\DeclareMathOperator{\dist}{dist}
\DeclareMathOperator{\Lin}{Lin}
\DeclareMathOperator{\Ker}{Ker}
\DeclareMathOperator{\diver}{div}
\DeclareMathOperator{\rank}{rank}
\DeclareMathOperator{\conv}{co}
\DeclareMathOperator{\diag}{diag}

\newcommand{\restr}{\mathrel{\tikz[baseline=0ex]{\draw[line width=.5pt,line cap=round] (0ex,1ex) -- (0ex,0ex) -- (1ex,0ex);}}}

\newcommand{\dd}{\;\mathrm{d}}

\numberwithin{equation}{section}

\title[Bounds on thin sets for measures satisfying PDEs]{Scale-invariant bounds on thin sets  for measures satisfying a first-order PDE and the anisotropic Michael--Simon inequality}

\author{Guido De Philippis}
\address{Department of Mathematics, University of Padua, Via Trieste 63, 35121 Padova, Italy}
\email{guido.dephilippis@math.unipd.it}

\author{Luca Gennaioli}
\address{Mathematics Institute, University of Warwick, Coventry CV4 7AL, UK}
\email{L.Gennaioli@warwick.ac.uk}

\author{Alessandro Pigati}
\address{Department of Decision Sciences, Bocconi University, Via Roentgen 1, 20136 Milano, Italy}
\email{alessandro.pigati@unibocconi.it}

\author{Filip Rindler}
\address{Mathematics Institute, University of Warwick, Coventry CV4 7AL, UK}
\email{F.Rindler@warwick.ac.uk}

\date{}

\begin{document}

\begin{abstract}
We prove mass estimates for PDE-constrained measures on sets of $\sigma$-finite ${\mathcal H}^m$-measure, assuming an appropriate rank condition. As an application, we derive a Michael--Simon inequality for varifolds with bounded anisotropic first variation with respect to an integrand that satisfies the natural atomic condition (AC1).
\end{abstract}

\maketitle

\section{Introduction}\label{sec:introduction}

The understanding of the interplay between PDE constraints and oscillation/concentration of (sequences of) functions is a central topic in PDE and the Calculus of Variations and dates back to the work of Tartar on compensated compactness \cite{Tar79}.  In the last decade there has been a renewed interest in it,
especially in the form of the singular part of PDE-constrained measures and in the behavior of sequences of functions. In particular, in the work~\cite{DR16} the first and fourth named  authors showed that the singular  part of a measure satisfying a PDE constraint is forced to lie in the \emph{wave cone} associated with that PDE operator. More precisely, specializing to the case of first-order operators,  it was shown in~\cite[Theorem 1.1]{DR16} that if \(V,W\) are  finite-dimensional vector spaces, \(\mathcal{A}: C^\infty(\mathbb{R}^d;V)\to C^\infty(\mathbb{R}^d;W)\) is a linear  differential operator,
\[
 \mathcal{A}:=\sum_{j=1}^d \A_j\partial_j, \quad \A_j\in \Lin(V,W),
\]
and  \(\sigma\in \mathcal{M}(\mathbb{R}^d;V)\) is a Radon measure such that
\[
 \mathcal{A} \sigma=\nu \in \mathcal{M}(\mathbb{R}^d;W),
\]
and \(E\) is such that \(\mathcal{L}^d(E)=0\), then
\[
 \frac{\mathrm{d} \sigma}{\mathrm{d} |\sigma|}(x)\in \Lambda_{\mathcal{A}}:=\bigcup_{|\xi|=1}\Ker \mathbb{A}(\xi),\quad   \mathbb{A}(\xi):=\sum_{j=1}^d \xi_j \A_j \in \Lin(V,W)
\]
for \(|\sigma|\)-a.e.\ \(x\in E\).   This  work was later extended in~\cite{ADHR19}, where it was shown that further restrictions are imposed on the polar \(\frac{\mathrm{d} \sigma}{\mathrm{d} |\sigma|}(x)\) on the part of \(\sigma\) which is singular with respect to lower-dimensional Hausdorff measures: If \(E\) is such that \(\mathcal{H}^m (E)=0\), then
\[
 \frac{\mathrm{d} \sigma}{\mathrm{d} |\sigma|}(x)\in \Lambda^m_{\mathcal{A}}:=\bigcap_{\pi \text{ \(m\)-plane}}\bigcup_{\xi\in\pi \setminus \{0\}}\Ker \mathbb{A}(\xi).
\]
In fact, in this statement one could replace the \(m\)-dimensional Hausdorff measure with the \(m\)-dimensional integral-geometric measure. In particular, by combining this with the Besicovitch--Federer projection theorem, one can recover various rectifiability results: see the surveys~\cite{DePRin18,DeP18} and the lecture notes~\cite{DeP21}. These ideas were also behind the main result in~\cite{DDG18}, where the authors identified a necessary and sufficient condition on the anisotropic integrand for the validity of  Allard's rectifiability theorem for  varifolds with controlled anisotropic first variation: see Section~\ref{sec:anis-integr-mich} below for precise statements.

Besides these qualitative results, it is a natural question whether they can be quantified. This was done in a simple perturbative regime in~\cite{ArrDePHirRinSko24}; see also~\cite{GueRaiSch24} and the recent work~\cite{BatWei25}. Apart from its intrinsic interest, a quantitative  enhancement of the results in~\cite{DR16} and~\cite{ADHR19} would have implications in recovering a-priori bounds on concentration measures for various PDEs. In particular, it was clear that these improvements would be a key step in proving a Michael--Simon-type estimate for varifolds with controlled first variation with respect to an anisotropic energy. This connection was used in~\cite{DP26} by the first and third named authors, building upon some unpublished ideas of Almgren: there, the validity of a Michael--Simon-type inequality for \(2\)-dimensional varifolds in \(\mathbb{R} ^3\) was established under the additional assumption that the anisotropic energy is sufficiently close to the area. The same strategy was later extended to \(2\)-dimensional varifolds in any codimension by Firester and Tsiamis in~\cite{FT26b}.

Recently, the second and fourth named authors proved (a generalized version of)  Bouchitt\'e's Vanishing Mass Conjecture~\cite{GR26}. Among the several  consequences of the main result in~\cite{GR26} is a quantitative improvement of the results in~\cite{DR16}, namely~\cite[Theorem~1.9]{GR26}. Soon after the appearance of~\cite{GR26}, Firester and  Tsiamis applied the main result of~\cite{GR26} to prove a Michael--Simon inequality for stationary varifolds in various situations, covering in particular all convex anisotropies in codimension one~\cite{FT26a}.

The goal of this work is to provide a quantitative bound for measures that satisfy a first-order PDE constraint, based on the techniques introduced in~\cite{GR26}. It can be thought of as a first step towards a quantification of the result in~\cite{ADHR19}. Furthermore, we show how this result implies the Michael--Simon inequality for varifolds with bounded first variation  with respect to anisotropic   energies satisfying a suitable (weak) atomic condition in all dimensions and codimensions.

\subsection*{Disclosure on AI usage.} An initial proof of Theorem~\ref{thm:support} was provided to the authors by a commercial AI model when asked if the techniques in \cite{GR26} could have been used to prove a sequential version of the result in~\cite{ADHR19} (in the spirit of the Vanishing Mass Conjecture). The model could not obtain it but, when asked to show Theorem \ref{thm:support}, it provided an essentially complete proof, by combining one of the main new analytic ingredients of~\cite{GR26}, namely~\cite[Lemma 3.2]{GR26}, with a log-Sobolev-type inequality. We later improved the result  by also showing higher integrability  of the density. For completeness, we report in~\cite{AI26} the draft provided by the AI model.  All the proofs in the paper   have been checked by the authors and they take full responsibility for them. AI was also used in the final stage to check for typos.

\subsection{Setting and main results}\label{sec:setting-main-results}

We first state our main result for the PDE operator $\mathcal{A}= \diver$, since it is the case which directly applies to the setting of varifolds with bounded anisotropic first variation.

\begin{theorem}\label{thm:support}
Let $0<m\leq d\le D$ be integers and let $\mathcal K\subset\R^{D\times d}$ be a (nonempty) compact
and convex set, with
\[
\lambda:=\min_{X\in\mathcal K}\sigma_m(X)>0,
\quad \Lambda:=\max_{X\in\mathcal K} \sigma_1(X) < \infty,
\]
where $\sigma_1(X)\ge\dots\ge\sigma_d(X)\ge0$ denote the singular values of $X$.
Let $M\subset\R^d$ be a Borel set with
$\HH^m(M)<\infty$ and let $\mu=\theta\HH^m\restr M$ be a finite
positive measure on $\R^d$. Suppose that $\sigma:=X\mu$, where $X:M\to\mathcal K$, and that $\diver \sigma$ (the row-wise divergence) is a finite measure. Then
\begin{equation}\label{eq:theorem}
\mu(\R^d)\le C(d,\lambda,\Lambda)\HH^m(M)^{1/m}|\diver \sigma|(\R^d).
\end{equation}
\end{theorem}

\begin{rmk}
Thanks to~\cite[Proposition~3.1]{ADHR19}, the measure $\mu$ and hence the set $M$ (up to replacing it with $M\cap\{\theta>0\}$) are automatically $m$-rectifiable.
\end{rmk}

Since, with the notation of the introduction, we have
\[
  \Lambda^m_{\diver}=\bigl\{X\in\R^{D\times d}:\rank X<m \bigr\}=\{X: \sigma_m(X)=0\},
\]
the above result can be thought of as a quantitative enhancement of~\cite{ADHR19}. Note also that the  condition that the matrix field $X$ is valued in $\mathcal{K}$, with constants $0 < \lambda \leq \Lambda < \infty$, is an ellipticity condition and enforces that the polar (and any of its averages) of the measure $\sigma$ lies in a set that has positive distance from the $m$-dimensional wave cone. Note that the convexity of $\mathcal{K}$ is as essential here as it is for~\cite[Theorem~1.9]{GR26}, because it entails that averages retain the same constraint.

The following generalization to arbitrary first-order operators $\A$ also holds and, in fact, it is easily seen to be equivalent to the previous statement.

\begin{theorem}\label{thm:support.bis}
Let $0<m\leq d$ be integers and consider a (nonempty) compact
and convex set $\mathcal K\subset V$, as well as a first-order operator $\A=\sum_{j=1}^d \A_j\partial_j$, where
$\A_j\in\Lin(V,W)$ are linear maps between Euclidean spaces. 
For $T_{\kappa}:\R^d\to W$ given by
\[
  T_{\kappa}(\xi):=\sum_j\xi_j\A_j\kappa, \quad \kappa\in\mathcal K,
\]
assume that
\[
\lambda:=\min_{\kappa\in\mathcal K}\sigma_m(T_\kappa)>0,
\quad \Lambda:=\max_{\kappa\in\mathcal K} \sigma_1(T_\kappa) < \infty.
\]
Letting $M\subset\R^d$ and $\mu$ be as above, and $X:M\to\mathcal K$, $\sigma:=X\mu$, suppose that $\A \sigma$ is a finite measure. Then
\begin{equation}\label{eq:theorem.bis}
\mu(\R^d)\le C(\A,d,\lambda,\Lambda)\HH^m(M)^{1/m}|\A \sigma|(\R^d).
\end{equation}
\end{theorem}
To see that Theorem~\ref{thm:support.bis} follows from Theorem~\ref{thm:support}, identify $W=\R^{D}$ and let $\mathrm{d}\tilde\sigma(x):=T_{X(x)}\dd\mu(x)$; equivalently, $\tilde\sigma$ is the matrix-valued measure with columns $\A_j \sigma$. Since $\diver\tilde\sigma=\A \sigma$, one concludes that the statement of Theorem~\ref{thm:support} in this case is Theorem~\ref{thm:support.bis}.


The above results can be improved to  deduce a higher integrability property for the density of $\mu$, namely the following.

\begin{theorem}\label{Lp}
    Assume that $\mu$ is a finite positive measure on $\R^d$ and that $\sigma:=X\mu$, where $X:\R^d\to\mathcal{K}$ is as in Theorem~\ref{thm:support.bis}. There exist $p\in(1,\infty)$ and $C>0$,
    depending on $d,\lambda,\Lambda$, such that the following holds:
    If $\A \sigma$ is a finite measure and $M$ is a Borel set with $\sigma$-finite $\HH^m$-measure then, writing $\mu\restr M=\theta\HH^m\restr M$, we have
    \[\|\theta\|_{L^{p,\infty}(\HH^m\restr M)}\le\frac{C}{r^{m(p-1)/p}}\bigl[\mu(\R^d)+r|\A \sigma|(\R^d)\bigr]^{1/p}\mu(\R^d)^{1-1/p}\]
    for all $r>0$.
\end{theorem}
\begin{rmk}
    By a standard layer-cake argument we also get uniform bounds on $\|\theta\|_{L^q(\HH^m\restr M)}$ for all $q<p$. Whether the $L^{p,\infty}$ bounds extend up to the endpoint $p=m/(m-1)$ is left as an open question.
\end{rmk}

\subsection{Anisotropic integrands and Michael--Simon inequality}\label{sec:anis-integr-mich}

As an application of Theorem~\ref{thm:support} we will extend the  Michael--Simon inequality in~\cite{MS73} to  varifolds with bounded anisotropic first variation. Let \(\mathrm{Gr}_m (\mathbb{R}^d)\) be the Grassmannian of \(m\)-planes in \(\mathbb{R} ^d\) and consider  an autonomous integrand $F:\mathrm{Gr}_m(\R^d)\to(0,\infty)$ of class $\mathrm{C}^1$. Define
\begin{equation}
    \label{eq:anis_area}
    \mathcal{F}(V):=\int_{\R^d} F(T_x V) \dd |V|(x),
\end{equation}
where $V$ is an $m$-rectifiable varifold and $T_xV$ is its tangent $m$-plane at $|V|$-a.e.\ point $x$.
The first variation of $\mathcal F$ is given by
\[
\langle\delta^FV,Y\rangle:=\int B_F(T_xV):DY(x) \dd |V|(x)
\]
for any vector field $Y\in \mathrm{C}_c^1(\R^d;\R^d)$, where $B_F:\mathrm{Gr}_m(\R^d)\to\R^{d\times d}$ is explicitly computable in terms of $F$~\cite[Lemma~A.2]{DDG18}. Note that, in the isotropic case of the usual area, where $F\equiv1$, $B_F(P)$ is simply the orthogonal projection onto $P$. In general, $B_F(P)$ has rank $m$ and its transpose, $B_F(P)^T$,  takes values in $P$ (while $\Ker B_F(P)=P^\perp$), but it need not be symmetric or positive semidefinite.

An energy \(F\) is said to satisfy the  \emph{atomic condition (AC)} if
\begin{itemize}
\item[(i)] the convex hull $\mathcal K_F:=\conv(\{B_F(P)\mid P\in\mathrm{Gr}_m(\R^d)\})$ consists of matrices of rank at least $m$;
\item[(ii)] non-extremal points of $\mathcal K_F$ have rank $>m$.
\end{itemize}
By itself, Part~(i) is referred to as the \emph{(AC1) condition}. The atomic condition was introduced in~\cite{DDG18} as a very natural ellipticity condition, in that the analogue of Allard's rectifiability criterion for varifolds (i.e., that varifolds $V$ with bounded anisotropic first variation are rectifiable on the set $\{x:\Theta^{m,*}(|V|,x)>0\}$) holds if and only if (AC) is in place (see~\cite[Theorem~1.2]{DDG18} and~\cite[Proposition~3.1]{ADHR19} for the improvement to positive upper density). We also note that in codimension \(1\) (AC1)  is equivalent to convexity of the integrand (once one identifies \(\mathrm{Gr}_{d-1}(\mathbb{R}^d)\) with \(\mathbb{S}^{d-1}/\{\pm1\}\) and extends the resulting even function $F:\mathbb{S}^{d-1}\to(0,\infty)$ one-homogeneously).

We then have the following result.

\begin{theorem}\label{ms}
Let $F$ be an anisotropic integrand that satisfies (AC1).
Then the Michael--Simon inequality
\begin{equation}\label{eq:piersilvio}
|V|(\R^d)\le C(F)\mathcal{H}^m(\{\Theta_V\ge1\})^{1/m}|\delta^FV|(\R^d)
\end{equation}
holds for all rectifiable varifolds $V$ with density $\Theta_V\ge1$ $|V|$-a.e.\ (with both $|V|$ and $|\delta^FV|$ finite measures). As a consequence, for $m>1$ we have
\begin{equation}
\label{eq:gianfranco}
|V|(\R^d)^{(m-1)/m}\le C(F)|\delta^FV|(\R^d).
\end{equation}
\end{theorem}

\begin{proof}
The rectifiable measure $\mathrm{d}\sigma(x):=B_F(T_xV)\dd|V|(x)$ satisfies $\diver \sigma=-\delta^FV$ and $B_F(T_xV)$ takes values in the compact convex set $\mathcal K_F$.
By (AC1), every matrix in $\mathcal K_F$ has rank at least $m$ and thus, since $\mathcal K_F$ is compact,
we have $\lambda=\min_{X\in\mathcal K_F}\sigma_m(X)>0$.

If $\Theta_V\ge1$ $|V|$-a.e., we may take $M:=\{\Theta_V\ge1\}$ and apply Theorem~\ref{thm:support}, which gives~\eqref{eq:piersilvio}. Noting that $\HH^m(M)\le |V|(\R^d)$ can be absorbed into the left-hand side, we also get~\eqref{eq:gianfranco}.
\end{proof}

In the setting of anisotropic integrands (see, for example,~\cite{D25} for a detailed account), the lack of a monotonicity formula makes the study of varifolds with bounded (anisotropic) first variation much more challenging. The validity of inequality~\eqref{eq:gianfranco} overcomes some of the issues. Indeed, assuming (AC), it allows one, for instance, to obtain a compactness theorem for rectifiable varifolds with bounded anisotropic first variation, and it is in fact equivalent to the latter~\cite[Proposition~1.9]{DP26}. We refer to the introduction of~\cite{DP26} for an equivalent functional version of the Michael--Simon inequality and for other useful consequences.

As already mentioned, recent progress on similar results to Theorem~\ref{ms} was made by the first and third named authors in~\cite{DP26} following unpublished work of Almgren, where such an inequality was obtained in the case $m=2$ and $d=3$ (hence codimension $1$) for integrands $F$ which are $\mathrm{C}^1$-close to the area functional. Subsequently, Firester and Tsiamis in~\cite{FT26b} extended this approach to $2$-varifolds in any ambient dimension $d$. The same authors, exploiting the resolution of the Vanishing Mass Conjecture by the second and fourth named authors in~\cite{GR26}, were able to extend inequality~\eqref{eq:gianfranco} to arbitrary dimension and codimension, retaining the assumption of $\mathrm{C}^1$-closeness to the area,  and to any convex anisotropy in codimension \(1\)~\cite[Theorem~1.1]{FT26a}. Such results have been more recently exploited to prove an anisotropic Allard regularity theorem in~\cite{DFT26}.  Theorem~\ref{ms} settles the validity of an anisotropic Michael--Simon inequality (and thus the compactness problem) in full generality, by proving it in arbitrary dimension and codimension without asking for closeness to the area functional, and under the natural (AC1) condition.

\subsection{Strategy of proof}\label{sec:strategy-proof}

For simplicity, we only discuss the strategy of proof of Theorem~\ref{thm:support},
which is strongly inspired by that of~\cite[Proposition~3.1]{GR26}; see also the AI disclosure statement. A tentative strategy to prove Theorem~\ref{thm:support} would be to restrict the measure $\sigma$ to pieces of $\mathrm{C}^1$ graphs and try to obtain good estimates on these pieces. Unfortunately, this is not a straightforward operation, as restricting the field $\sigma$ to Borel sets does not preserve the PDE constraint (i.e., the bound on the divergence) and leads to additional difficulties (much like in the original proof of the singular density theorem~\cite{DR16}, or~\cite{ADHR19}). Nevertheless, in this spirit, we restrict the ``weight'' $\mu$ while keeping the same ``polar'' $X$: namely, we introduce a better-controlled positive measure $\mu_0$ such that $\mu_0\leq\mu$, where $\mu$ is the one in the statement of Theorem~\ref{thm:support}; it is actually given by $\mu_0:=\mu\restr E$, where $E$ is an $m$-dimensional set with finite Minkowski content. The point is that we are still able to obtain useful bounds for such a $\mu_0$, even if $X\mu_0$ does not have finite divergence, exploiting the tame $m$-dimensional nature of $\mu_0$
and the fact that $\sigma=X\mu$ obeys a PDE constraint. Once this is done, a simple approximation argument allows us to infer a bound for the original measure $\mu$.

The crucial part of the argument is to obtain bounds on the \emph{entropy dissipation} for the ``heated'' $\mu_0$. One can indeed define for $t\in[0,T)$ the density
\[
  \rho_t:=\mathcal{P}_{T-t}\mu_0
\]
and estimate the rate of growth of its entropy (which increases in $t$ under the conventions used here). This rate is expressed in terms of the \emph{Fisher information} (which is the ``derivative'' of the entropy functional). Roughly speaking, the control on the divergence of $\sigma$ allows for an integral bound on the Fisher information
\[
  I(\rho_t):=-\int_{\R^d}\rho_t\tr\nabla^2\log\rho_t\dd x.
\]
The point is that, after using a trivial upper bound for $-\nabla^2\log\rho_t$ along the $d-m$ directions on which we have no control, we are able to get effective bounds along the $m$ tangent directions by exploiting the PDE.

Finally, we can reinterpret the previous integral bound as an upper bound for the total gap
\[
  \limsup_{t\to T^-}\tilde H(\rho_t)-\tilde H(\rho_0),
\]
where
\[
  \tilde H(\rho_t):=H(\rho_t)+\frac{(d-m)\mu_0(\R^d)}{2}\log(T-t)
\]
is a renormalized entropy adapted to the diffusion of $m$-dimensional measures. While for $\tilde H(\rho_0)$ we can use a trivial estimate, the final part of the argument involves obtaining a lower bound on $\limsup_{t\to T^-}\tilde H(\rho_t)$ (roughly speaking, the renormalized entropy of $\mu_0$) using the size of the support of $\mu_0$ (together with the Minkowski content estimate).


\subsection*{Acknowledgements.} The authors' research is funded by the European Research Council (ERC) and the UK Research and Innovation (UKRI) agency: GDP acknowledges support from CoG 101169953 ``RISE'', LG and FR acknowledge ERC/UKRI grant EP/Z000297/1 ``CONCENTRATE'', and AP acknowledges StG 101165368 ``MAGNETIC''.

\section{Notation}\label{sec:notation}
We write $|\nu|$ for the total variation measure associated with a (Borel) vector-valued measure $\nu$ on $\R^d$. Similarly, $|V|$ denotes the weight of a varifold $V$, i.e., the projection of $V$ onto the base space $\R^d$.
The divergence of a matrix-valued measure is taken row-wise, i.e.,
$(\diver\nu)_i:=\sum_{j=1}^d\partial_j\nu_{ij}$ as distributions.

For rectangular matrices of the same size, we use the Hilbert--Schmidt scalar product, namely
\[\langle A,B\rangle=A:B:=\tr(A^TB),\quad |A|^2=A:A.\]

For sets $E\subseteq\R^d$, we will write $B_r(E):=\{x:\dist(x,E)<r\}$ to denote the $r$-neighborhood.
Given a measure space $(X,\mu)$ and $p\in[1,\infty)$, let us recall that the weak-$L^p$ quasinorm of a measurable function $f:X\to\R$ is given by
\[\|f\|_{L^{p,\infty}(\mu)}
:=\sup_{u>0}u\mu(\{|f|>u\})^{1/p}.\]

We will often apply the heat flow semigroup to compactly supported measures $\nu$: letting
\[
p_s:\R^d\to(0,\infty),\quad p_s(x):=(4\pi s)^{-d/2}e^{-|x|^2/(4s)}
\]
denote the usual fundamental solution of the heat equation, for any $s>0$ we define
\[\PP_s\nu:=p_s*\nu.\]
Given a smooth $f:\R^d\to(0,\infty)$ with sufficiently fast decay at infinity (e.g., faster than $|x|^{-d-\varepsilon}$), we define its \emph{entropy} to be
\[H(f):=\int_{\R^d}f\log f\dd x.\]
Note that we use the convention making entropy increase along the \emph{backward} heat flow. We will compute entropies of functions which do not necessarily satisfy $\int_{\R^d}f\dd x=1$.

\section{Heat estimates}\label{sec:heat-estimates}

In this section we explain how to adapt the bounds from~\cite[Section~3]{GR26} for our purposes, while keeping the exposition self-contained. For simplicity, let us assume that $|X|\le\Lambda$ (up to enlarging $\Lambda$).
Moreover, we assume throughout the section that $\mu$ has compact support and $\mu(\R^d)=1$.
Let us set $\delta:=|\diver \sigma|(\R^d)$ and fix $T>0$. For $0\le t<T$, using the forward heat semigroup $(\PP_s)_{s>0}$, we define
\begin{equation}\label{eq:notation}
U_t:=\PP_{T-t}\mu,\quad V_t:=\PP_{T-t} \sigma,\quad R_t:=\PP_{T-t}|\diver \sigma|,\quad
S_t:=U_t+\sqrt{T-t}\,R_t,
\end{equation}
as well as
\begin{equation}
B_t:=\frac{V_t}{S_t},\quad \zeta_t:=\nabla\log S_t=\frac{\nabla S_t}{S_t},\quad
\eta_t:=\frac{\PP_{T-t}\diver \sigma}{S_t},\quad
\kappa_t:=\frac{R_t}{2\sqrt{T-t}\,S_t}.
\end{equation}
We will often drop the subscript $t$.

In particular, we have
\begin{equation}\label{eq:source}
\HH S=-\kappa S,\quad \HH V=0,\quad
\int_{\R^d} S_t\dd x=1+\sqrt{T-t}\,\delta,
\end{equation}
where $\HH:=\partial_t+\Delta$, while by convexity
\begin{equation}\label{cvx}
\frac{V}{U}\in\mathcal K,\quad |B|\le \Lambda.
\end{equation}


For a (nonzero) submeasure $0\le\mu_0\le\mu$, set
\[
\rho_t:=\PP_{T-t}\mu_0,\quad \gamma:=\mu_0(\R^d),
\]
and note that $0<\gamma\le1$, as well as $0\le\rho\le U\le S$.

The following lemma replaces~\cite[Lemma~3.2]{GR26}; compared to it,
besides taking $a\equiv 1$, we use $\rho$ in place of $S$. Moreover,
as opposed to the argument in~\cite[Section~3]{GR26}, where $\rho$ gets dissipated
on the region where $q(B)\ge-\varepsilon$, here $\rho$ simply evolves through the backward heat semigroup.

\begin{lemma}\label{lem:energy}
There exists a constant $C(d,\Lambda)$ such that
\begin{equation}\label{eq:selected}
\int_0^T\int_{\R^d}\rho
\bigl(|\nabla B|^2+|B\nabla\log\rho|^2\bigr) \dd x \dd t
\le C\gamma+C\gamma\log\biggl(\frac{1+\sqrt{T}\,\delta}{\gamma}\biggr).
\end{equation}
\end{lemma}

\begin{proof}
For any smooth $\varphi$ we first compute that
\begin{equation}\label{eq:product}
\HH(S\varphi)=S\HH_\zeta\varphi-\kappa S\varphi,
\end{equation}
where $\HH_\zeta:=\HH+2\zeta\cdot\nabla$, so that $\HH_\zeta B=\kappa B$.

Letting
\[
  J(t):=\int_{\R^d}\rho_t\log(\rho_t/S_t)\dd x,
\]
using Jensen's inequality for $s\mapsto s\log s$, and recalling that $\rho\le S$,
we get
\begin{equation}\label{jensen}
\gamma\log\biggl(\frac \gamma{1+\sqrt{T-t}\,\delta}\biggr)\le J(t)\le0.
\end{equation}
Recalling that $\HH\rho=0$ and using~\eqref{eq:product} with $\varphi=q:=\rho/S$ we get again $\HH_\zeta q=\kappa q$.
Hence, by the chain rule, $\Phi(q):=q\log q$ is such that
\[\HH_\zeta\Phi(q)=\Phi'(q)\HH_\zeta q+\Phi''(q)|\nabla q|^2
=(1+\log q)\kappa q+\frac{|\nabla q|^2}{q},\]
so that
\[
  \HH(\rho\log(\rho/S))=\HH(S\Phi(q))=S\HH_\zeta\Phi(q)-\kappa S\Phi(q)=\kappa\rho+\rho|\nabla\log q|^2
\]
and thus
\begin{equation}\label{eq:relative}
J'(t)=\int_{\R^d}\rho_t\biggl|\nabla\log\frac{\rho_t}{S_t}\biggr|^2 \dd x
+\int_{\R^d}\kappa_t\rho_t \dd x.
\end{equation}
It follows that
\begin{equation}\label{pluto}
\int_0^T\int_{\R^d}\rho\biggl|\nabla\log\frac\rho S\biggr|^2 \dd x \dd t
\le -J(0) \leq \gamma\log\biggl(\frac{1+\sqrt T\,\delta}{\gamma}\biggr).
\end{equation}
We now compute that
\[\HH(\rho|B|^2)=2\rho|\nabla B|^2+2\langle\nabla\rho-\rho\zeta,\nabla|B|^2\rangle+2\kappa\rho|B|^2.\]
Integrating in space and using $|\nabla|B|^2|\le2\Lambda|\nabla B|$
(recall that $|B|\le\Lambda$), we obtain
\[\frac{\mathrm{d}}{\mathrm{d}t}\int_{\R^d}\rho_t|B_t|^2 \dd x\ge\int_{\R^d}2\rho_t|\nabla B_t|^2 \dd x-\int_{\R^d}4\Lambda\rho_t|\nabla\log(\rho_t/S_t)||\nabla B_t| \dd x
.\]
By Young's inequality, since $\int_{\R^d}\rho_t|B_t|^2 \dd x\le\Lambda^2\int_{\R^d}\rho_t \dd x=\Lambda^2\gamma$ we deduce that
\[\int_0^T\int_{\R^d}\rho|\nabla B|^2 \dd x \dd t\le\Lambda^2\gamma+C\int_0^T\int_{\R^d}\rho|\nabla\log(\rho/S)|^2 \dd x \dd t,\]
which gives
\begin{equation}\label{nabla.B.est}
    \int_0^T\int_{\R^d}\rho|\nabla B|^2 \dd x \dd t\le C\gamma+C\gamma\log\biggl(\frac{1+\sqrt{T}\,\delta}{\gamma}\biggr)
\end{equation}
thanks to~\eqref{pluto}.

Next, by construction we have $|\eta_t|\le\frac{1}{\sqrt{T-t}}$,
so that
\[\int_{\R^d}\rho_t|\eta_t|^2 \dd x\le\frac{1}{T-t}\int_{\R^d}\rho_t \dd x=\frac{\gamma}{T-t}\]
and
\[
\int_{\R^d} \rho_t|\eta_t|^2 \dd x\le\int_{\R^d}S_t|\eta_t|^2 \dd x
\le\frac{1}{\sqrt{T-t}}\int_{\R^d}|\PP_{T-t}\diver \sigma| \dd x\le\frac{\delta}{\sqrt{T-t}}.
\]
If $\sqrt{T}>\gamma/\delta$ then, letting $s:=T-t$, we can estimate
\[\int_0^T\int_{\R^d}\rho|\eta|^2 \dd x \dd s
\le\int_0^{\gamma^2/\delta^2}\frac{\delta}{\sqrt{s}} \dd s
+\int_{\gamma^2/\delta^2}^T\frac{\gamma}{s}\dd s
=2\gamma+2\gamma\log\biggl(\frac{\sqrt{T}\,\delta}{\gamma}\biggr),\]
while if $\sqrt{T}\le\gamma/\delta$ then the second bound gives directly that
\[
  \int_0^T\int_{\R^d}\rho|\eta|^2\dd x\dd s\le2\sqrt{T}\,\delta,
\]
so that in both cases we have
\begin{equation}\label{eta.est}
    \int_0^T\int_{\R^d}\rho|\eta|^2 \dd x \dd t
    \le 4\gamma\log\biggl(1+\frac{\sqrt{T}\,\delta}{\gamma}\biggr).
\end{equation}
Also,~\cite[eq.~(3.9)]{GR26} is unchanged and reads
$\diver B+B\zeta=\eta$. Together with~\eqref{nabla.B.est} and~\eqref{eta.est}, this gives
\begin{equation}\label{eq:fullenergy}
\int_0^T\int_{\R^d} \rho|B\zeta|^2 \dd x \dd t
\le C\gamma+C\gamma\log\biggl(\frac{1+\sqrt{T}\,\delta}{\gamma}\biggr).
\end{equation}
Since $|B|\le \Lambda$ and
$B\nabla\log\rho=B\zeta+B\nabla\log(\rho/S)$, combining~\eqref{eq:fullenergy} with~\eqref{pluto} we also get the desired integral bound for $\rho|B\nabla\log\rho|^2$.
\end{proof}

We now apply the previous bounds in order to obtain an upper bound
on (a renormalized version of) the Fisher information, which in turn
bounds the entropy growth along the evolution of $\rho_t$.

\begin{lemma}\label{lem:dimensional}
For every $0<\tau<T$, writing
$I(\rho_t):=\int_{\R^d}\rho_t|\nabla\log\rho_t|^2\dd x$, we have
\begin{equation}\label{eq:dimensional}
\int_0^\tau\biggl[I(\rho_t)-\frac{(d-m)\gamma}{2(T-t)}\biggr] \dd t
\le C\gamma+C\gamma\log\biggl(\frac{1+\sqrt{T}\,\delta}{\gamma}\biggr)
\end{equation}
for some constant $C(d,\lambda,\Lambda)$.
\end{lemma}

\begin{proof}
Choose a smooth $\chi:[0,\infty)\to[0,1]$, equal to zero on
$[0,\lambda^2/16]$ and to one on $[\lambda^2/4,\infty)$. Let
$Q:=\chi(B^TB)$ be defined by spectral calculus as
\[
  Q:=O\diag(\chi(\sigma_1^2(B)),\dots,\chi(\sigma_d^2(B)))O^T,
\]
where we write $B^TB=O\diag(\sigma_1^2(B),\dots,\sigma_d^2(B))O^T$ using the spectral theorem. Moreover, set $\Pi:=Q^2$, which plays the role of a projection matrix. We have
\begin{equation}\label{eq:spectral}
0\le\Pi\le I_d,\quad |Qz|^2\le C|Bz|^2,\quad
|\nabla Q|\le C|\nabla B|.
\end{equation}
On $\{U_t\ge\sqrt{T-t}\,R_t\}$,~\eqref{eq:notation} and~\eqref{cvx} imply
$\sigma_m(B_t)\ge \lambda/2$. To see this, observe that $S_t\leq 2U_t$ gives $U_t/S_t\geq1/2$, and thanks to~\eqref{cvx} we can write
\[
  \sigma_m(B_t)=\frac{U_t}{S_t}\sigma_m(V_t/U_t)\geq\lambda/2.
\]
Hence, $\tr\Pi_t\ge m$ by construction of $\chi$. The complementary
set $E_t$ has
\[
  \int_{E_t}\rho_t\dd x\le\int_{E_t} U_t\dd x\le\int_{E_t}\sqrt{T-t}\,R_t\dd x\le\sqrt{T-t}\,\delta.
\]
Therefore,
\begin{equation}\label{eq:trace}
\int_{\R^d}\rho_t\tr(I_d-\Pi_t)\dd x\le(d-m)\gamma+d\min\{\gamma,\sqrt{T-t}\,\delta\}.
\end{equation}

Differentiating the convolution with the Gaussian $p_s$ (where $s:=T-t$)
and noting that
\[
  \nabla^2 p_s(x-y)=\biggl[\frac{(x-y)\otimes (x-y)}{4s^2}-\frac{I_d}{2s}\biggr]p_s(x-y),
\]
we get
\begin{equation}\label{eq:hessian}
\nabla^2\log\rho_t\ge-\frac{I_d}{2(T-t)}.
\end{equation}
This is the well-known log semi-convexity improvement of the heat flow, and can be proved by observing that the difference of the two terms equals the covariance of $y$ under the probability
measure $p_s(x-y)\dd\mu_0(y)/\rho_t(x)$, divided by $4s^2$. Using $I(\rho_t)=-\int_{\R^d}\rho_t\tr(\nabla^2\log\rho_t)\dd x$,
\eqref{eq:trace}, and~\eqref{eq:hessian}, we furthermore obtain
\begin{align*}
I(\rho_t)&=-\int_{\R^d}\rho_t\tr((I_d-\Pi_t)\nabla^2\log\rho_t) \dd x-\int_{\R^d}\rho_t\tr(\Pi_t\nabla^2\log\rho_t) \dd x\\
&\le\frac{(d-m)\gamma}{2(T-t)}
+\min\biggl[\frac{d\gamma}{2(T-t)},\frac{d\delta}{2\sqrt{T-t}}\biggr]
-\sum_i\int_{\R^d}\rho_t [e_i^T(\nabla^2\log\rho_t)e_i] \dd x,
\end{align*}
where $e_i=e_{i,t}$ denotes the $i$-th column of $Q_t$.

Now, for every smooth vector field $e$ with compact support, integration
by parts gives
\begin{equation}\label{eq:ibp}
-\int_{\R^d}\rho_t [e^T(\nabla^2\log\rho_t)e] \dd x
=\int_{\R^d}\rho_t [
(\diver e+e\cdot\nabla\log\rho_t)^2
-\tr((De)^2)] \dd x,
\end{equation}
where $(De)_{jk}:=\partial_ke_j$.
Indeed, writing $b_t:=\diver e+e\cdot\nabla\log\rho_t=\rho^{-1}_t\diver(\rho_t e)$, we have
$\int_{\R^d}\rho_t b_t^2 \dd x=-\int_{\R^d}\rho_t e\cdot\nabla b_t \dd x$ and
\[e\cdot\nabla b_t=e\cdot\nabla \diver e+((e\cdot\nabla)e)\cdot\nabla\log\rho_t+e^T(\nabla^2\log\rho_t)e,\]
and, once we integrate, the first two terms on the right-hand side combine to give the trace term $-\int_{\R^d}\rho_t\tr((De)^2)\dd x$. By a standard cut-off argument,
\eqref{eq:ibp} also applies to the columns $e_{i,t}$ of $Q_t$.
Summing over $i$ and recalling~\eqref{eq:spectral}, this gives
\begin{equation}\label{eq:fisher}
I(\rho_t)\le\frac{(d-m)\gamma}{2(T-t)}+\min\biggl[\frac{d\gamma}{2(T-t)},\frac{d\delta}{2\sqrt{T-t}}\biggr]
+C\int_{\R^d}\rho_t(|\nabla B_t|^2+
|B_t\nabla\log\rho_t|^2) \dd x.
\end{equation}
Integrating in $t$, applying Lemma~\ref{lem:energy},
and bounding the integral of the second term as in the previous proof, we reach the claim.
\end{proof}

\section{Proofs of the main theorems}\label{sec:proofs-main-theorems}

\begin{proof}[Proof of Theorem~\ref{thm:support}]
Assume again that $\mu$ has compact support, $\mu(\R^d)=1$, $|\diver \sigma|(\R^d)=\delta$, and $0\le\mu_0\le\mu$ with $\gamma:=\mu_0(\R^d)\in(0,1]$;
the submeasure $\mu_0$ will be chosen in a moment.

The idea now is to look at the entropy $H(\rho_t)$ as time passes.
More specifically, we look at a renormalized version of it, adapted to codimension $d-m$:
let
\[\tilde H(\rho_t):=H(\rho_t)+\frac{(d-m)\gamma}{2}\log(T-t)=\int_{\R^d}\rho_t\log\rho_t \dd x+\frac{(d-m)\gamma}{2}\log(T-t).\]
Since $\HH\rho=0$, the derivative $\frac{\mathrm{d}}{\mathrm{d}t}H(\rho_t)=I(\rho_t)\ge0$ is the Fisher information (recall that
$\rho_t$ evolves through the backward heat equation), and thus
\[\frac{\mathrm{d}}{\mathrm{d}t}\tilde H(\rho_t)=I(\rho_t)-\frac{(d-m)\gamma}{2(T-t)}.\]
The previous lemma then gives
\[\limsup_{t\to T^-}\tilde H(\rho_t)-\tilde H(\rho_0)\le C\gamma+C\gamma\log\biggl(\frac{1+\sqrt T\,\delta}{\gamma}\biggr).\]
Moreover, since $\rho_0=\PP_T\mu_0$, we have $\rho_0\le\frac{\gamma}{(4\pi T)^{d/2}}$, and hence
\[H(\rho_0)\le\int_{\R^d}\rho_0\log\biggl(\frac{\gamma}{(4\pi T)^{d/2}}\biggr)\dd x\le\gamma\log\gamma-\frac{d\gamma}{2}\log(4\pi T),\]
giving in particular
\[\tilde H(\rho_0)\le-\frac{m\gamma}{2}\log T.\]
Thus, for some constant $C(d,\lambda,\Lambda)$, we get
\begin{equation}\label{u.b.for.h}
\limsup_{t\to T^-}\tilde H(\rho_t)\le -\frac{m\gamma}{2}\log T+C\gamma+C\gamma\log\biggl(\frac{1+\sqrt{T}\,\delta}{\gamma}\biggr).
\end{equation}

Next, we observe that $\mu$ is a rectifiable measure, by~\cite[Proposition~3.1]{ADHR19}.
We can then take $E\subseteq M$ to be a finite disjoint union of compact sets, each of which is included in a $1$-Lipschitz graph (up to a rotation), and $\mu_0:=\mu\restr E$.
We claim a lower bound on the left-hand side, namely
\begin{equation}\label{lb}
\liminf_{t\to T^-}\tilde H(\rho_t)\ge\gamma\log\biggl(\frac{\gamma}{\HH^m(E)}\biggr)-C(d)\gamma.
\end{equation}
Note that the geometry of $E$ is not controlled, namely $E$ could start resembling an $m$-plane only
at very small scales, but the key point is that in the limit $t\to T$ only the picture at smaller and smaller scales matters. More precisely, we will just need a bound on the Minkowski content at (non-uniformly) small scales.

To establish~\eqref{lb}, we first compare the probability density $\varphi_t:=\rho_t/\gamma$
with a more explicit one, namely
\[
 \psi_s:= \gamma_s^{-1}
 e^{-\dist(x,E)^2/(4s)},\quad\gamma_s:=\int_{\R^d}e^{-\dist(x,E)^2/(4s)}\dd x,
\]
where we set $s:=T-t$ for simplicity. Since the relative entropy $\int_{\R^d}\varphi_t(\log\varphi_t-\log\psi_t)\dd x\ge0$,
after a little rearrangement we get
\[H(\rho_t)\ge\gamma(\log\gamma-\log\gamma_s)-\frac{1}{4s}\int_{\R^d}\dist(x,E)^2 \rho_t(x)\dd x.\]
Since $\dist(x,E)\le|x-y|$ for all $y\in E$ and $\rho_t=\PP_s\mu_0$, the last integral is bounded by
\[
  \int_E\int_{\R^d}|x-y|^2p_s(x-y)\dd x\dd\mu_0(y)=2d\gamma s,
\]
so that
\[
  H(\rho_t)\ge\gamma(\log\gamma-\log\gamma_s)-\frac{d\gamma}{2}.
\]
Finally, by construction of $E$, it is straightforward to check that there exists $r_0>0$ such that $|B_r(E)|\le C(d)r^{d-m}\HH^m(E)$ for all radii $0<r<r_0$,
giving
\begin{align*}\gamma_s&=\int_0^\infty\frac{r}{2s}e^{-r^2/(4s)}|B_r(E)|\dd r\\
&\le C(d)\int_0^{r_0}\frac{r^{1+d-m}}{2s}e^{-r^2/(4s)}\HH^m(E)\dd r+C(E)e^{-r_0^2/(8s)}\\
&\le C(d)s^{(d-m)/2}\HH^m(E)\end{align*}
for $s$ small enough, proving~\eqref{lb}.

Combining~\eqref{u.b.for.h} with~\eqref{lb},
we arrive at
\begin{equation}\label{final.bd}
\log\biggl(\frac{\gamma}{\HH^m(E)}\biggr)\le -\frac{m}{2}\log T+C+C\log\biggl(\frac{1+\sqrt{T}\,\delta}{\gamma}\biggr).
\end{equation}
This implies that $\delta>0$, since otherwise we would reach a contradiction for $T$ large enough. Noting that we can take $E$ as above such that $\gamma=\mu_0(\R^d)=\mu(E)\in[1/2,1]$,
we can then take $T:=\delta^{-2}$, which yields
\[\log\HH^m(E)+\log\delta^m\ge-C\]
for some constant $C(d,\lambda,\Lambda)$,
and hence
\[\mu(\R^d)=1\le C\HH^m(E)^{1/m}\delta=C\HH^m(E)^{1/m}|\diver \sigma|(\R^d),\]
as desired.

By homogeneity, this proves the theorem for a compactly supported $\mu$,
while the general case follows by a standard cut-off argument.
\end{proof}

\begin{proof}[Proof of Theorem~\ref{Lp}]
Without loss of generality, we can assume that $\A=\diver$, as in the statement of Theorem~\ref{thm:support}.
Moreover, as in the previous proof, we can assume that $\mu$ is compactly supported and $\mu(\R^d)=1$ (by homogeneity).

We take again $\mu_0:=\mu\restr E$
with $E\subseteq M$ a finite disjoint union of compact pieces of $1$-Lipschitz graphs.
We observe that, taking $T:=r^2$, \eqref{final.bd} rearranges to
\[0\le\log\biggl(\frac{\HH^m(E)}{r^m}\biggr)+C'+C'\log\biggl(\frac{1+r\delta}{\gamma}\biggr)\]
for some $C'(d,\lambda,\Lambda)$; since $\gamma=\mu(E)$, taking $\varepsilon:=1/C'$ this is the same as
\[\frac{\mu(E)}{1+r\delta}\le C\biggl(\frac{\HH^m(E)}{r^m}\biggr)^\varepsilon.\]
By approximation, this holds for any $E\subseteq M$ with $\HH^m(E)\in(0,\infty)$ (assuming without loss of generality that $M$ is $m$-rectifiable).
Given $u>0$, taking $E_u:=\{x\in M:\theta(x)>u\}$ this yields
\[u\HH^m(E_u)
\le\int_{E_u}\theta\dd\HH^m
\le\mu(E_u)
\le C\frac{u^{-\varepsilon}}{r^{m\varepsilon}}(1+r\delta),\]
where we used the fact that $\HH^m(E_u)\le u^{-1}\mu(E_u)\le u^{-1}$.
Taking $p:=1+\varepsilon$, this is the same as
\[u^p\HH^m(\{x\in M:\theta(x)>u\})\le \frac{C}{r^{m(p-1)}}(1+r\delta)\]
for any given $u>0$, as desired.
\end{proof}

\end{document}